\documentclass[intlimits]{article}
\usepackage{latexsym,amsfonts,amssymb,amsmath,amsthm,amscd}
\usepackage{chngcntr}

\newtheorem{theorem}{Theorem}[part]

\newtheorem{proposition}[theorem]{Proposition}

\newtheorem{problem}[theorem]{Problem}

\newtheorem{remark}[theorem]{Remark}

\counterwithin{section}{part}

\numberwithin{equation}{part}

\def\p{\partial}

\newcommand\JM{Mierczy\'nski}
\newcommand\RR{\ensuremath{\mathbb{R}}}

\newcommand\NN{\ensuremath{\mathbb{N}}}

\newcommand{\PP}{\ensuremath{\mathbb{P}}}

\newcommand{\OFP}{\ensuremath{(\Omega,\mathcal{F},\mathbf{P})}}

\newcommand{\lambdadom}{\ensuremath{\lambda_{\mathrm{dom}}}}

\newcommand{\abs}[1]{\ensuremath{\lvert#1\rvert}}

\newcommand{\norm}[1]{\ensuremath{\lVert#1\rVert}}

\DeclareMathOperator{\diag}{diag}

\DeclareMathOperator{\trace}{tr}
\DeclareMathOperator{\perm}{perm}

\begin{document}

\title{Estimates for principal Lyapunov exponents: \\
A survey}

\author{Janusz \JM
\\
Institute of Mathematics and Computer Science
\\
Wroc{\l}aw University of Technology
}
\date{}
\maketitle

\abstract{This is a survey of known results on estimating the principal Lyapunov exponent of a time-dependent linear differential equation possessing some monotonicity properties.  Equations considered are mainly strongly cooperative systems of ordinary differential equations and parabolic partial differential equations of second order.  The estimates are given either in terms of the principal (dominant) eigenvalue of some derived time-independent equation or in terms of the parameters of the equation itself.  Extensions to other differential equations are considered.  Possible directions of further research are hinted.}

\part{Introduction}
It is well known that for {\em autonomous\/} linear systems of ordinary differential equations (ODEs)
\begin{equation}
\label{eq:ODE-intro}
\frac{du}{dt} = C u,
\end{equation}
where $C$ is a real $N$ by $N$ matrix, as~well~as for {\em autonomous\/} linear problems for parabolic second order partial differential equations (PDEs)
\begin{equation}
\label{eq:PDE-intro}
\begin{cases}
\displaystyle \frac{\partial u}{\partial t}  = {\Delta} u + c(x) u, &
x \in D
\\[1ex]
u(x) = 0, &  x \in \partial D,
\end{cases}
\end{equation}
where $D \subset \RR^n$ is a bounded domain with sufficiently smooth boundary $\partial D$, $\Delta = \partial^2/\partial_1^2 {}+ \ldots + {} \partial^2/\partial_n^2$ is the Laplacian and $c \colon \overline{D} \to \RR$ is sufficiently regular, asymptotic behavior of solutions to \eqref{eq:ODE-intro}, respectively to~\eqref{eq:PDE-intro}, is fully
determined by the spectrum of the matrix $C$, respectively by the spectrum of the elliptic boundary value problem
\begin{equation}
\label{eq:PDE-intro-elliptic}
\begin{cases}
0 = {\Delta} u + c(x) u, & x \in D
\\
u(x) = 0, &  x \in \partial D.
\end{cases}
\end{equation}
More precisely, the spectrum of the solution operator at $t > 0$ equals,
\begin{itemize}
\item
in the case of~\eqref{eq:ODE-intro}, the exponent of the spectrum of $tC$;
\item
in the case of~\eqref{eq:PDE-intro}, the exponent of $t$ times the spectrum of the problem~\eqref{eq:PDE-intro-elliptic}, plus $0$.
\end{itemize}

Indeed, the above property, called the Spectral Mapping Theorem, holds for a wide class of semigroups of linear operators (see, e.g., \cite{EngNag}).

\medskip
For a system of {\em nonlinear autonomous\/} ordinary differential equations
\begin{equation}
\label{eq:ODE-intro-nonlinear}
\frac{dz}{dt} = f(z),
\end{equation}
where $f = (f_1, \dots, f_N) \colon \RR^N \to \RR^N$ is a $C^1$ vector field such that $f(0) = 0$, the linear part of the solution operator around $0$ equals precisely the solution operator of the linearization of $f$ at $0$, that is, $e^{tC}$, where
\begin{equation*}
C = \biggl[ \frac{\partial f_i}{\partial z_j}(0) \biggr]_{i,j=1}^{N}.
\end{equation*}
Similarly, in the case of {\em nonlinear autonomous\/} parabolic PDE
\begin{equation}
\label{eq:PDE-intro-nonlinear}
\begin{cases}
\displaystyle \frac{\partial z}{\partial t} = {\Delta} z + f(x, z), &
x \in D,
\\[1ex]
z(t,x) = 0, &  x \in \partial D,
\end{cases}
\end{equation}
where $f \colon \overline{D} \times \RR \to \RR$ is a sufficiently smooth function such that $f(x, 0) = 0$, the linear part of the solution operator of~\eqref{eq:PDE-intro-nonlinear} around $0$ is precisely the solution operator of the linearization of \eqref{eq:PDE-intro-nonlinear} at $0$, that is, of the problem~\eqref{eq:ODE-intro} with
\begin{equation*}
c(x) = \frac{\partial f}{\partial z}(x, 0), \quad x \in \overline{D}.
\end{equation*}

\smallskip
Usually the above two steps are combined as the Principle of Linearized Stability, stating that if all eigenvalues of the linearization of~\eqref{eq:ODE-intro-nonlinear} or of~\eqref{eq:PDE-intro-nonlinear} at the trivial solution have negative real parts then the trivial solution is (exponentially) asymptotically stable, and if there is an eigenvalue with positive real part then the trivial solution is unstable.

\medskip
There arises a question what happens in the case when the original nonlinear differential equation is nonautonomous, or even when it is autonomous but its solution whose (in)stability we are investigating is not stationary.  The linear part (derivative) of the solution operator around that reference solution is given by the solution of
the {\em nonautonomous\/} linear equation being the linearization of the original equation at the reference solution.

\smallskip
In order not to go into unnecessary generality, consider the trivial solution of a periodic (with period $T > 0$) system of ODEs
\begin{equation}
\label{eq:ODE-nonlinear}
\frac{dz}{dt} = f(t, z),
\end{equation}
where $f(t + T, z) = f(z)$ for all $t \in \RR$ and $z \in \RR^N$ and $f(t, 0) = 0$ for all $t$.

After linearization we obtain
\begin{equation}
\label{eq:ODE-intro-periodic}
\frac{du}{dt} = C(t) u,
\end{equation}
where
\begin{equation}
\label{eq:ODE-linearization}
C(t) = \Bigl[ \frac{\partial f_i}{\partial z_j}(t,0) \Bigr]_{i,j=1}^{N}.
\end{equation}
Let $U(t; s)$ stand for the transition matrix of~\eqref{eq:ODE-intro-periodic}:  $U(t;s)u_0$ is the value at time $t$ of the solution of the initial value problem
\begin{equation}
\label{eq:IVP-intro-periodic}
\begin{cases}
\displaystyle \frac{du}{dt} = C(t) u,
\\[1ex]
u(s) = u_0.
\end{cases}
\end{equation}

The family $\{ U(t; s) \}$ satisfies the following equalities
\begin{equation*}
U(t; r) = U(t; s) U(s; r) \quad \text{for any } r, s, t \in \RR,
\end{equation*}
and
\begin{equation*}
U(t + T; s + T) = U(t; s) \quad \text{for any } s, t \in \RR.
\end{equation*}
Each $t > 0$ can be written as $kT + \theta$ for some $k \in \NN \cup \{0\}$ and $\theta \in [0, T)$, so we have $U(t;0) = U(\theta;0) U(T;0)^k$.  Consequently, in order to investigate asymptotic behavior of solutions of~\eqref{eq:ODE-intro-periodic} it suffices to investigate the iterates of $U(T; 0)$.  The Principle of Linearized Stability works for time\nobreakdash-\hspace{0pt}periodic systems of ODEs, too: if the spectral radius of $U(T;0)$ is smaller than one then the trivial solution of~\eqref{eq:ODE-nonlinear} is (exponentially) asymptotically stable, and if the spectral radius of $U(T;0)$ is larger than one then the trivial solution of~\eqref{eq:ODE-nonlinear} is unstable.

\smallskip
Now it becomes clear that it would be good to obtain estimates of the spectrum of the matrix $U(T;0)$ in terms of (perhaps the spectra of) its ``generators,'' that is, of the matrices $C(t)$, $t \in [0, T]$.  However, there is a caveat: whereas in the autonomous case for each $t > 0$ there is the relation
\begin{equation*}
U(t; 0) = \exp{( t C)},
\end{equation*}
in the $T$\nobreakdash-\hspace{0pt}periodic case, for $N > 1$ the relation
\begin{equation*}
U(T; 0) = \exp{\biggl( \int\limits_{0}^{T} C(t) \, dt \biggr)}
\end{equation*}
need not hold (cf.\ \cite{Iser}).

In the light of the above one is hardly surprised to learn that numerous examples are known of ODE systems for which all eigenvalues of $C(t)$ have negative real parts but the spectral radius of $U(T;0)$ is greater than one, as well as ODE systems for which some eigenvalues of $C(t)$ have positive real parts but the spectral radius of $U(T;0)$ is smaller than one (see a nice paper~\cite{JR}).

\bigskip
It is the purpose of the present paper to give a (rather incomplete) survey of what is known regarding (upper and lower) estimates of the spectral radius of the solution operator for some special cases of (systems of) linear time\nobreakdash-\hspace{0pt}periodic differential equations.  The unifying feature of all the equations considered here is that their solution operators have some (strong) order\nobreakdash-\hspace{0pt}preserving property.  For such an equation there exists a unique (up to multiplication by positive scalars) positive solution, $v(t)$, whose directions attract, as $t \to \infty$, the directions of any other positive solution.  The logarithmic growth rate of $v(t)$ (called the {\em principal Lyapunov exponent\/}) equals, firstly, the logarithmic growth rate of the solution operator, and, secondly, the logarithmic growth rate of any positive solution.

\medskip
Whereas in many applications to physical sciences or engineering stability (and, consequently, an \textbf{upper} estimate of the spectrum) is looked~for, the situation in biological sciences is frequently different.  Let us look at some such applications.
We restrict ourselves to time\nobreakdash-\hspace{0pt}periodic ODE systems
\begin{equation*}
\frac{dz}{dt} = f(t, z)
\end{equation*}
that are {\em strongly cooperative\/}, that is, $\p f_{i}/\p z_j > 0$ for $i \ne j$.  Moreover, assume that $f(t,0) = 0$ for all $t \in \RR$.

Systems of the above form appear when modeling populations in which individuals can transit between several states (for instance, bacterial populations can switch between dormant and active states, see~\cite{MalSm}), or populations in which individuals can move between discrete patches (for discrete\nobreakdash-\hspace{0pt}time models, see~\cite{Schreib}).  If $z_{i}(t)$ measures the population size in state $i$ (or in patch $i$), strong cooperativity is a natural assumption (whereas the mortality and growth rates in state $i$ are reflected in the derivative $\p f_i/\p z_i$).  In many cases we are interested in the permanence of the population, which mathematically means that the zero solution should be in some sense repelling.  The positivity of the principal Lyapunov exponent of the linearization is a natural sufficient condition for that.  Therefore finding viable \textbf{lower} estimates of the principal Lyapunov exponent is of immense interest.

\medskip
The paper is organized as~follows.  In Part~\ref{part:ODEs} we give a survey of results on strongly cooperative linear systems of ODEs, and in Part~\ref{part:PDEs} we consider linear second\nobreakdash-\hspace{0pt}order PDEs of parabolic type.  We strive to highlight symmetry in both cases.

In Part~\ref{part:generalizations-non-periodic} generalizations are mentioned to almost\nobreakdash-\hspace{0pt}periodic, general nonautonomous, and random systems.

Part~\ref{part:generalization-other-DE} deals with extensions to other systems:  systems of linear ODEs that are only cooperative, strongly cooperative systems of reaction--diffusion equations, general parabolic PDEs of second order, nonlocal dispersal equations, and, finally, cooperative systems of delay differential equations.

Part~\ref{part:discrete-time} is devoted to discrete\nobreakdash-\hspace{0pt}time systems, more precisely to random systems of positive matrices.  Only a couple of results are mentioned, and in the context of their continuous\nobreakdash-\hspace{0pt}time analogs.

\medskip
The majority of theorems presented (or at~least alluded to) in the present paper give bounds on the principal Lyapunov exponent in terms of the dominant (or principal) eigenvalues for some time\nobreakdash-\hspace{0pt}independent system derived from the original one, either by taking minimum or maximum (for~example, Proposition~\ref{prop:ODE-trivial}) or by taking time\nobreakdash-\hspace{0pt}average (for~example, Theorem~\ref{thm:PDE-averaging}).

Some results, such as Theorem~\ref{thm:Kolotilina-Frobenius} or a Faber--Krahn type inequality in Subsection~\ref{subsect:Faber-Krahn}, are different in the sense that they provide estimates of the principal Lyapunov exponent directly in terms of the parameters of the system.

\medskip
Regarding proofs, usually hints or indications are given only, with references to existing papers.  In some instances a fairly complete proof is given, as is the case with Theorem~\ref{thm:averaging-ODEs}, which is a specialization of a much more general (and much more difficult to prove) result.

\bigskip
Throughout the paper, $I$ stands for the identity mapping (operator) or the identity matrix, depending on the context.  For a matrix $C$, $C^{\top}$ denotes its transpose.

\part{Cooperative Systems of ODEs}
\label{part:ODEs}
In Part~\ref{part:ODEs},  $\langle \cdot, \cdot \rangle$ denotes the standard inner product, and $\norm{\cdot}$ denotes the Euclidean norm, in $\RR^N$.

\section{Preliminaries}
\label{sect:preliminaries-ODEs}

\subsection{Preliminaries: ODEs}
\label{subsect:preliminaries-ODEs}
For a system of linear ordinary differential equations
\begin{equation}
\label{eq:ODEs}
\frac{du}{dt} = C(t) u,
\end{equation}
where $C \colon \RR \to \RR^{N \times N}$ is a continuous and
\textbf{time\nobreakdash-\hspace{0pt}periodic} (with period $T$) matrix function, denote by $U(t; s) u_0$ ($s, t \in \RR$, $u_0 \in \RR^N$) the value at time $t$ of the solution of \eqref{eq:ODEs} satisfying the initial condition $U(s; s) u_0 = u_0$.  $U(t;s)$ is called the {\em transition matrix\/} ({\em Cauchy matrix\/}, {\em propagator\/}).

The family $\{ U(t; s) \}_{s, t \in \RR}$ of $N$ by $N$ matrices satisfies the following equalities:
\begin{equation*}
U(t; r) = U(t; s) U(s; r),
\end{equation*}
and
\begin{equation*}
U(t + T; s + T) = U(t; s).
\end{equation*}

We call $U(T; 0)$ the {\em monodromy matrix\/} of~\eqref{eq:ODEs} (in the context of nonlinear systems we usually speak of the {\em Poincar\'e map\/}, or {\em period map}).

\subsection{Preliminaries: Perron Theorem}
\label{subsect:preliminaries-Perron}
We introduce the following notation:
\begin{gather*}
\RR^N_{+} := \{\, u \in \RR^N: u_i \ge 0 \text{ for all } i = 1, 2,
\dots, N \,\},
\\
\RR^N_{++} := \{\, u \in \RR^N: u_i > 0 \text{ for all }i = 1, 2,
\dots, N\,\}.
\end{gather*}
Vectors in $\RR^N_{+}$ (resp.\ in $\RR^N_{++}$) will be called {\em nonnegative\/} (resp.\ {\em positive\/}). Similarly, matrices with all entries nonnegative (resp.\ positive) will be referred to as {\em nonnegative\/} (resp.\ {\em positive\/}) {\em matrices\/}.

We recall the celebrated theorem of Perron (sometimes called the Perron--Frobenius theorem; see, e.g., \cite[Theorem (1.4)]{BP}, or \cite[Theorem 1.1]{Sen}).
\begin{theorem}[Perron(--Frobenius)]
\label{thm:Perron} The spectral radius of a positive $N$ by $N$ matrix $C$ has the following properties:
\begin{enumerate}
\item[{\rm (i)}]
it is positive, and is a simple eigenvalue;
\item[{\rm (ii)}]
an eigenvector pertaining to it can be chosen positive;
\item[{\rm (iii)}]
all the remaining eigenvalues of $C$ have moduli smaller than the spectral radius;
\item[{\rm (iv)}]
eigenvectors pertaining to the remaining eigenvalues do not belong to $\RR^{N}_{+}$.
\end{enumerate}
\end{theorem}
By $\mathcal{M}_{N}$ we denote the set of all $N$ by $N$ matrices with positive off\nobreakdash-\hspace{0pt}diagonal entries.  Since for $C \in \mathcal{M}_{N}$ its exponential $e^{tC}$ is, for all $t > 0$, a positive matrix, a direct application of the Spectral Mapping Theorem for matrices yields the following.
\begin{theorem}
\label{thm:Perron-off-diagonal}
For $C \in \mathcal{M}_{N}$, its eigenvalue with the largest real part \textup{(the dominant eigenvalue)} has the following properties:
\begin{enumerate}
\item[{\rm (i)}]
it is real and simple;
\item[{\rm (ii)}]
an eigenvector pertaining to it can be chosen positive;
\item[{\rm (iii)}]
all the remaining eigenvalues of $C$ have real parts smaller than the dominant eigenvalue;
\item[{\rm (iv)}]
eigenvectors pertaining to the remaining eigenvalues do not belong to $\RR^{N}_{+}$.
\end{enumerate}
\end{theorem}
We will denote the dominant eigenvalue of $C \in \mathcal{M}_{N}$ by $\lambdadom(C)$.

\section{Principal Lyapunov Exponent for \\ Strongly Cooperative Systems of ODEs}
\label{sect:strongly-cooperative-ODEs}
We call a linear time\nobreakdash-\hspace{0pt}periodic system of ODEs
\begin{equation}
\label{eq:ODE}
\frac{du}{dt} = C(t) u
\end{equation}
{\em strongly cooperative} if for any $t \in \RR$ the matrix $C(t)$ is in $\mathcal{M}_{N}$.

It follows from a theorem for quasimonotone systems of (not~necessarily linear) ODEs due to M\"uller and Kamke (see, e.g., \cite[Chapter II]{W}, or \cite[Chapter 3]{HiSm}) that the monodromy matrix $U(T;0)$ for a strongly cooperative \eqref{eq:ODE} is positive.

A solution $u \colon [0,\infty) \to \RR^N$ of \eqref{eq:ODE} is called {\em positive\/} if $u(t) \in \RR^N_{++}$ for all $t > 0$.  A necessary and sufficient condition for a solution of strongly cooperative~\eqref{eq:ODE} to be positive is that $u(0) \in \RR^{N}_{+} \setminus \{0\}$.

A solution $u \colon \RR \to \RR^N$ is called {\em globally positive\/} if $u(t) \in \RR^N_{++}$ for all $t \in \RR$.

\begin{theorem}
\label{thm:Floquet-ODE}
For a linear $T$\nobreakdash-\hspace{0pt}periodic strongly cooperative system \eqref{eq:ODE} there exist and are uniquely determined: a number $\rho > 0$, and a globally positive solution $v(t)$ of \textup{(\ref{eq:ODE})} such that
\begin{itemize}
\item[{\rm (i)}]
$v(T) = {\rho} v(0)$;
\item[{\rm (ii)}]
$\norm{v(0)} = 1$.
\end{itemize}
Consequently, the {\em logarithmic growth rate\/} of $v(t)$, $\lim\limits_{t \to \infty} \frac{1}{t}\ln\norm{v(t)}$, exists and is equal to $\frac{1}{T}\ln\rho$.  Moreover,
\begin{itemize}
\item[{\rm (iii)}]
\begin{equation*}
\lim\limits_{t \to \infty} \frac{\ln{\norm{U(t;0)}}}{t} = \frac{\ln{\rho}}{T};
\end{equation*}
\item[{\rm (iv)}]
for any positive solution $u(t)$ of \textup{(\ref{eq:ODE})} there holds
\begin{equation*}
\lim\limits_{t \to \infty} \frac{\ln{\norm{u(t)}}}{t} = \frac{\ln{\rho}}{T}.
\end{equation*}
\end{itemize}
\end{theorem}
In the above theorem as $v(t)$ we take the solution of~\eqref{eq:ODE} passing at time zero through the normalized positive eigenvector $v$ of the monodromy matrix $U(T;0)$, and $\rho > 0$ is just the spectral radius of $U(T;0)$.  Part (iii) is a quite straightforward consequence of Theorem~\ref{thm:Perron}.  Part (iv) as~well~as the uniqueness (up~to multiplication by positive scalars) of the globally positive solution $v(t)$ also follow from Theorem~\ref{thm:Perron}, but their proof requires more effort.

The number
\begin{equation*}
\Lambda := \frac{\ln{\rho}}{T},
\end{equation*}
appearing in the statement of Theorem~\ref{thm:Floquet-ODE} is called the {\em principal Lyapunov exponent\/} (sometimes {\em principal Floquet exponent\/}) of \eqref{eq:ODE}.

\medskip
A reader acquainted with the Floquet theory of periodic linear ODEs (see, e.g., \cite[Section 2.4]{Chic}, or~\cite[Section 2.2]{Far}) will notice that $\rho$ is the characteristic multiplier of~\eqref{eq:ODE} with the largest modulus and $\Lambda$ is the Floquet exponent with the largest real part (it is true that the imaginary parts of Floquet exponents are determined up~to addition of an integer multiple of $2\pi$, but in our case we can put the imaginary part to be equal to zero).

\bigskip
We gather now a couple of formulas that will be useful in obtaining estimates of the principal Lyapunov exponent.

A simple application of the Chain Rule gives that
\begin{equation}
\label{eq:ODE-lambda}
\Lambda = \frac{1}{T} \int\limits_{0}^{T} \frac{\langle C(t) v(t),
v(t) \rangle}{\norm{v(t)}^2} \, dt.
\end{equation}

\smallskip
It is easy to see that the function $w \colon \RR \to \RR^{N}$, where $w(t) := v(t) / \norm{v(t)}$, is the solution of the system
\begin{equation*}
\frac{du}{dt} = (C(t) - \kappa(t) I) u
\end{equation*}
satisfying $w(0) = v$, where $\kappa(\cdot)$ is the integrand in~\eqref{eq:ODE-lambda}.

\section{Principal Lyapunov Exponents for ODE Systems: ``Trivial'' Estimate}
\label{section:ODE-trivial}
For a linear $T$-periodic strongly cooperative system \eqref{eq:ODE}, denote by $C_{\mathrm{min}}$ (resp.\ $C_{\mathrm{max}}$) the matrix whose $(i,j)$\nobreakdash-\hspace{0pt}entry equals the minimum (resp.\ maximum) of $c_{ij}(t)$ over $t \in [0, T]$.
\begin{proposition}
\label{prop:ODE-trivial}
For a linear $T$\nobreakdash-\hspace{0pt}periodic strongly cooperative system \eqref{eq:ODE} there holds
\begin{equation*}
\lambda_{\mathrm{dom}}(C_{\mathrm{min}}) \le \Lambda \le \lambda_{\mathrm{dom}}(C_{\mathrm{max}}).
\end{equation*}
\end{proposition}
\begin{proof}[Indication of proof]
For each $t \in \RR$ there holds
\begin{equation*}
C_{\mathrm{min}} v(t) \le v'(t) \le C_{\mathrm{max}} v(t),
\end{equation*}
where the inequalities are understood componentwise.  It follows by standard comparison theorems for quasimonotone systems of ordinary differential inequalities (see, e.g., \cite[Chapter II]{W}) that
\begin{equation*}
\exp(t C_{\mathrm{min}}) v(0) \le v(t) \le \exp(t C_{\mathrm{max}}) v(0), \quad t > 0.
\end{equation*}
Since the logarithmic growth rate of a positive solution of $dv/dt = C_{\mathrm{min}} v$ equals $\lambda_{\mathrm{dom}}(C_{\mathrm{min}})$ (see~Theorem~\ref{thm:Floquet-ODE}(iv)) and an analogous property holds for $dv/dt = C_{\mathrm{max}} v$, the statement follows.
\end{proof}

\section{Principal Lyapunov Exponents for ODE Systems: Estimates from Above}
\label{section:ODE-upper-estimates}
For an $N \times N$ real matrix $C$ denote by $C^{\mathrm{S}}$ its {\em symmetrization\/}:
\begin{equation*}
C^{\mathrm{S}} := \frac{C + C^{\top}}{2}.
\end{equation*}

If $C$ is in $\mathcal{M}_{N}$ then $C^{\mathrm{S}}$ is a symmetric matrix in $\mathcal{M}_{N}$.  For any $u \in \RR^{N}$ there holds
\begin{equation}
\label{eq:ODE-symmetrisation}
\langle C u, u \rangle = \langle \tfrac{1}{2} C u, u \rangle +
\langle u, \tfrac{1}{2} C u \rangle = \langle C^{\mathrm{S}} u, u
\rangle \le \lambdadom(C^{\mathrm{S}}) \: \norm{u}^2.
\end{equation}
By substituting $v(t)$ for $u$ in \eqref{eq:ODE-symmetrisation} we obtain from \eqref{eq:ODE-lambda} the following.
\begin{proposition}
\label{prop:ODEs-above}
For a linear $T$\nobreakdash-\hspace{0pt}periodic strongly cooperative system \eqref{eq:ODE} there holds
\begin{equation*}
\Lambda \le \frac{1}{T} \int\limits_{0}^{T}
\lambdadom(C^{\mathrm{S}}(t)) \, dt.
\end{equation*}
\end{proposition}
We should mention here that the inequalities on which the proof of Proposition~\ref{prop:ODEs-above} is based appeared first (for general $C(t)$) in~\cite{Wint}.

We end by observing that by an inequality due to Bendixson (see~\cite[III.1.2.1]{MaMi}), for  $C \in \mathcal{M}_{N}$ one has $\lambdadom(C) \le \lambdadom(C^{\mathrm{S}})$, hence we would rather not hope, in~general, to obtain upper estimates of $\Lambda$ in~terms of the dominant eigenvalues of $C(t)$.

\section{Principal Lyapunov Exponents for ODE Systems: Estimates from Below}
\label{section:ODE-lower-estimates}
In the present section we give a survey of known results on \textbf{lower} estimates of the principal Lyapunov exponent for strongly cooperative systems of linear ODEs.

\subsection{Estimates in Terms of the Entries of a Matrix}
\label{subsection:ODE-Lyapunov-function}
For \textbf{time\nobreakdash-\hspace{0pt}independent} matrices with positive (or, more
general, nonnegative) off\nobreakdash-\hspace{0pt}diagonal entries there are plenty of results giving lower estimates of their dominant eigenvalues in terms of elementary functions of the entries of the matrix (see, e.g., \cite[pp.~152--158]{MaMi}). It is unclear to the author, however, how the existing proofs of those estimates could carry over to strongly cooperative nonautonomous linear systems of ODEs.

\medskip
Regarding the latter, to the author's knowledge the only result known so far is contained in his paper~\cite{Mi} (but see a discrete\nobreakdash-\hspace{0pt}time analog in~\cite{ProtJung}).  The idea is to find a homogeneous polynomial $V$ of the coordinates $u_1, \dots,u_N$ of $u$, taking positive values on $\RR^{N}_{++}$, with the
property that
\begin{itemize}
\item
there exists $\alpha > 0$ such that
\begin{equation*}
\norm{u} \ge \alpha V(u) \quad \text{for all } u \in \RR^{N}_{++};
\end{equation*}
\item
there is a continuous function $\beta \colon \mathcal{M}_N \to \RR$ such that for any $C \in \mathcal{M}_N$ the inequality
\begin{equation*}
\langle \nabla V(u), Cu \rangle \ge \beta(C) \, V(u)
\end{equation*}
holds for all $u \in \RR^{N}_{++}$.
\end{itemize}

\begin{theorem}
\label{thm:Kolotilina-Frobenius}
For a linear $T$\nobreakdash-\hspace{0pt}periodic strongly cooperative system \eqref{eq:ODE} there holds
\begin{enumerate}
\item[{\rm (i)}]
\begin{equation*}
\Lambda \ge \frac{1}{NT} \int\limits_{0}^{T} \Bigl( \trace{C(t)} + 2 \sum\limits_{j < k} \sqrt{c_{jk}(t) c_{kj}(t)} \,\Bigr) \, dt.
\end{equation*}
\item[{\rm (ii)}]
\begin{equation*}
\Lambda \ge \frac{1}{T} \int\limits_{0}^{T}
\Bigl( \min\limits_{i} \sum\limits_{j} c_{ij}(t) \Bigr) \, dt.
\end{equation*}
\item[{\rm (iii)}]
\begin{equation*}
\Lambda \ge \frac{1}{T} \int\limits_{0}^{T}
\Bigl( \min\limits_{j} \sum\limits_{i} c_{ij}(t) \Bigr) \, dt.
\end{equation*}
\end{enumerate}
\end{theorem}
\begin{proof}[Indication of proof]
(i) follows by applying the function $V(u) = u_1 \cdot \ldots \cdot u_N$, (ii) follows by applying the function $V(u) = u_1 + \ldots + u_N$, whereas (iii) is just a modification of (ii) to the adjoint system.
\end{proof}
A time\nobreakdash-\hspace{0pt}independent version of (i) (with a different proof, and for nonnegative matrices) appeared as Corollary~3 in~\cite{Kol}, and time\nobreakdash-\hspace{0pt}independent versions of (ii) and (ii) are classical Frobenius estimates (\cite[III.3.1.1]{MaMi}).

\begin{remark}
{\em Theorem~\ref{thm:Kolotilina-Frobenius} cannot be reduced to applying those known estimates for matrices.  Indeed, assume for~simplicity that $C(t)$ is, for all $t \in \RR$, a symmetric matrix.  Applying~\cite[Corollary~3]{Kol} to $C(t)$ for each individual $t \in \RR$ we have
\begin{equation*}
\lambda_{\mathrm{princ}}(C(t)) \ge \frac{1}{N} \Bigl( \trace{C(t)} + 2 \sum\limits_{j < k} c_{jk}(t) \Bigr),
\end{equation*}
which is in no conceivable relation with
\begin{equation*}
\Lambda \ge \frac{1}{NT} \int\limits_{0}^{T} \Bigl( \trace{C(t)} + 2 \sum\limits_{j < k} c_{jk}(t) \,\Bigr) \, dt.
\end{equation*}}
\end{remark}

\begin{problem}
{\em \textbf{For each estimate of the dominant eigenvalue of a matrix in $\mathcal{M}_N$ given in terms of its entries find a proof by means of a suitably chosen function $V$} (cf.\ also Problem~\ref{pr:2}).}
\end{problem}

\subsection{Estimates via Averaging}
\label{subsection:ODE-averaging}
In the present subsection we formulate a result giving an estimate of the principal eigenvalue in terms of the dominant eigenvalue of the time\nobreakdash-\hspace{0pt}averaged equation.  As its proof employs only elementary means, we give it here. It is a specialization of a (much more general) result contained in~\cite{MiSh}.

\smallskip
Throughout the present subsection we assume that $C \colon \RR \to \mathcal{M}_{N}$ is continuous and $T$\nobreakdash-\hspace{0pt}periodic and, moreover, the
off\nobreakdash-\hspace{0pt}diagonal entries are independent of time.  We denote by $\overline{C} = [\overline{c}_{ij}]$ the time average of $C$:
\begin{equation*}
\overline{c}_{ii} := \frac{1}{T} \int\limits_{0}^{T} c_{ii}(t) \, dt,
\qquad \overline{c}_{ij} := c_{ij} \text{ for } i \ne j.
\end{equation*}

\begin{theorem}
\label{thm:averaging-ODEs}
For a linear $T$\nobreakdash-\hspace{0pt}periodic strongly cooperative system~\textup(\ref{eq:ODE}) with off\nobreakdash-\hspace{0pt}diagonal entries independent of time we have
\begin{equation*}
\Lambda \ge \lambdadom(\overline{C}).
\end{equation*}
\end{theorem}
\begin{proof}
Recall that $w(\cdot) = (w_{1}(\cdot), \dots, w_{N}(\cdot))$ denotes
the ($T$\nobreakdash-\hspace{0pt}periodic) solution of the system
\begin{equation}
\label{eq:system-normalized}
\frac{du}{dt} = C(t) u - \kappa(t) u
\end{equation}
satisfying the initial condition $w(0) = v$, where $\kappa(t) = \langle C(t) v(t), v(t) \rangle / \norm{v(t)}^2$.

For any $i \ne j$ application of Jensen's inequality to the function
$w_{j}(\cdot)/w_{i}(\cdot)$ on the interval $[0, T]$ gives that
\begin{equation}
\label{eq:Jensen}
\exp\biggl(  \frac{1}{T} \int\limits_{0}^{T}
\ln\dfrac{w_{j}(t)}{w_{i}(t)} \, dt \biggr) \le \frac{1}{T}
\int\limits_{0}^{T} \frac{w_{j}(t)}{w_{i}(t)} \, dt.
\end{equation}
Denote
\begin{equation*}
\hat{w}_{i} := \exp\biggl(  \frac{1}{T} \int\limits_{0}^{T}
\ln{w_{i}(t)}\, dt \biggr),
\end{equation*}
and similarly for $\hat{w}_{j}$.  We can write \eqref{eq:Jensen} in
the form:
\begin{equation*}
\frac{\hat{w}_{j}}{\hat{w}_{i}} \le \frac{1}{T} \int\limits_{0}^{T}
\frac{w_{j}(t)}{w_{i}(t)} \, dt,
\end{equation*}
from which it follows that, for each $i$,
\begin{equation*}
\frac{1}{\hat{w}_{i}} \sum\limits_{j \ne i} c_{ij} \hat{w}_{j} \le
\frac{1}{T} \int\limits_{0}^{T} \Bigl( \frac{1}{w_{i}(t)}
\sum\limits_{j \ne i} c_{ij} w_{j}(t) \Bigr) \, dt.
\end{equation*}
As $w(\cdot)$ is a solution of \eqref{eq:system-normalized}, we have
\begin{equation*}
\frac{1}{v_{i}(t)} \sum\limits_{j \ne i} c_{ij} v_{j}(t) =
\frac{v'_{i}(t)}{v_{i}(t)} - (c_{ii}(t) - \kappa(t)).
\end{equation*}
Consequently,
\begin{equation*}
\sum\limits_{j \ne i} c_{ij} \hat{w}_{j} \le \biggl( \frac{1}{T}
\int\limits_{0}^{T} \frac{w'_{i}(t)}{w_{i}(t)} \, dt + \frac{1}{T}
\int\limits_{0}^{T} \bigl( c_{ii}(t) - \kappa(t) \bigr) \, dt
\biggl) \hat{w}_{i}.
\end{equation*}
Since $w_{i}(T) = w_{i}(0)$, the first summand on the right\nobreakdash-\hspace{0pt}hand side reduces to zero, so we obtain
\begin{equation}
\label{eq:averaged}
\sum\limits_{j \ne i} c_{ij} \hat{w}_{j} - \biggl( \frac{1}{T}
\int\limits_{0}^{T} \bigl( c_{ii}(t) - \kappa(t) \bigr) \, dt \biggr)
\hat{w}_{i} \le 0,
\end{equation}
that is, via~\eqref{eq:ODE-lambda},
\begin{equation*}
(\overline{C} - \Lambda I) \hat{w} \le 0,
\end{equation*}
where $\hat{w}$ is a positive vector. By \cite[Theorem (1.11) on p.~28]{BP}, the dominant eigenvalue of $\overline{C} - \Lambda I$ is nonpositive, that is, $\Lambda \ge \lambdadom(\overline{C})$.
\end{proof}

\part{Parabolic PDEs of Second Order}
\label{part:PDEs}

Throughout the whole Part~\ref{part:PDEs}, $D \subset \RR^n$ is a bounded domain, with boundary $\partial D$ of class $C^3$.  By $\overline{D}$ we denote the closure of $D$, $\overline{D} = D \cup \partial D$.

$L_2(D)$ stands for the Hilbert space of square\nobreakdash-\hspace{0pt}summable \textbf{real} functions on $D$, with the standard inner product $\langle \cdot, \cdot \rangle$ and norm $\norm{\cdot}$.

\section{Preliminaries}
\label{sect:preliminaries-PDEs}

\subsection{Preliminaries: Elliptic PDEs}
\label{subsect:preliminaries-elliptic-PDEs}
In the present subsection we collect some properties of elliptic second order operators of the form $\Delta + c(\cdot) I$ endowed with the Dirichlet boundary conditions, where $\Delta$ is the Laplacian, $\Delta = \partial^2/\partial x_1^2 {} + \ldots + {} \partial^2/\partial x_n^2$, and $c$ belongs to $L_{\infty}(D)$.  Such operators are sometimes called {\em Schr\"odinger operators\/} (on a bounded domain).

By an {\em eigenvalue\/} of the elliptic problem
\begin{equation}
\label{eq:elliptic}
\begin{cases}
{\Delta} u + c(x) u = 0, & x \in D,
\\
u(x) = 0, & x \in \partial D,
\end{cases}
\end{equation}
we understand a complex number $\lambda$ such that there exists a nontrivial (that is, not equal identically to zero) solution of the boundary value problem
\begin{equation*}
\begin{cases}
{\Delta} u + c(x) u = {\lambda} u, & x \in D,
\\
u(x) = 0, & x \in \partial D.
\end{cases}
\end{equation*}
Such a nontrivial solution is called an {\em eigenfunction\/} of~\eqref{eq:elliptic} corresponding to the eigenvalue $\lambda$.  It should be mentioned here that in the literature on PDEs, calculus of variations, etc., by an eigenvalue of~\eqref{eq:elliptic} a complex number $\lambda$ is understood such that there exists a nontrivial solution of ${\Delta} u + (c(x) + \lambda) u = 0$ (plus boundary condition).

It is well known (see, e.g., the classical book~\cite{CourHilb}) that all eigenvalues of \eqref{eq:elliptic} are real, simple and have $-\infty$ as their unique accumulation point. Moreover, an eigenfunction corresponding to the largest eigenvalue (the {\em principal eigenvalue\/} of \eqref{eq:elliptic}) can be chosen to take positive values on $D$.  We will denote the principal eigenvalue of the problem~\eqref{eq:elliptic} by $\lambda_{\mathrm{princ}}(\Delta + c(\cdot)I)$.

Let $H^{1}_{0}(D)$ stand for the Sobolev space of $L_2(D)$ functions whose (distributional) derivatives are in $L_2(D)$ and whose trace on the boundary $\p D$ is zero (for definitions see, e.g., \cite{Evans}).  The well\nobreakdash-\hspace{0pt}known variational characterization of the principal eigenvalue of the Schr\"odinger operator gives us the following.
\begin{proposition}
\label{prop:Rayleigh}
For \textup{(\ref{eq:elliptic})} we have
\begin{equation*}
\int\limits_{D} \left( - \norm{{\nabla} u(x)}_{\RR^n}^2  + c(x) u^2(x) \right) \, dx \le \lambda_{\mathrm{princ}}(\Delta + c(\cdot)I) \int\limits_{D} u^2(x) \, dx, \quad \forall \, u \in H_{0}^{1}(D),
\end{equation*}
where $\norm{\cdot}_{\RR^n}$ denotes the Euclidean norm on $\RR^n$, and the equality holds if~and~only~if $u \equiv 0$ or $u$ is a principal eigenfunction.
\end{proposition}

\subsection{Preliminaries: Time-Periodic Parabolic PDEs}
\label{subsect:preliminaries-parabolic-PDEs}
Consider a linear parabolic partial differential equation (PDE) of
second order
\begin{equation}
\label{eq:PDE-parabolic}
\begin{cases}
\displaystyle \frac{\partial u}{\partial t} = {\Delta} u + c(t,x) u, & x \in D
\\[1ex]
u(t, x) = 0, & x \in \partial D
\end{cases}
\end{equation}
where $c \colon \RR \times \overline{D} \to \RR$ is a $C^3$ function,
\textbf{time\nobreakdash-\hspace{0pt}periodic} in $t$ with period $T$.

Standard theorems on existence and uniqueness of solutions of linear parabolic PDEs (see, e.g., \cite{F}) state that for any initial moment $s \in \RR$ and any initial value $u_0$, if $u_0$ has second derivatives that are H\"older continuous on $\overline{D}$ and satisfies the boundary condition then there exists a unique function $u(t,x; s, u_0)$, $t \ge s$, $x \in \overline{D}$, such that
\begin{itemize}
\item[(S1)]
the function $\bigl[\, [s, \infty) \times \overline{D} \ni (t, x) \mapsto u(t,x; s, u_0) \in \RR \,\bigr]$ itself, its first and second derivatives in $x$ and its first derivative in $t$ are H\"older continuous, uniformly on sets of the form $[s, S] \times \overline{D}$, $S > s$;
\item[(S2)]
the equation in~\eqref{eq:PDE-parabolic} is satisfied pointwise for any $t > s$ and any $x \in D$, and the boundary condition in~\eqref{eq:PDE-parabolic} is satisfied pointwise for any $t > s$ and any $x \in \p D$ (that is, it is a {\em classical solution\/});
\item[(S3)]
$u(s, x; s, u_0) = u(x)$ for each $x \in D$.
\end{itemize}
However, the Banach spaces of H\"older continuous functions are rather unwieldy to work in.  The Hilbert space $L_2(D)$ is (at~least for our purposes) a better choice, especially because we will need to differentiate (in time) the norm of a solution.  Indeed, for an initial value $u_0$ we can take any function in $L_2(D)$.  Then there exists a unique function $u(t,x; s, u_0)$, $t > s$, $x \in \overline{D}$, such that (S1) and (S2) are satisfied with $s$ replaced with any $s_1 > s$, and the initial condition $u(s, \cdot) = u_0$ is understood in the following sense: $u(t, \cdot; s, u_0)$ converges, as $t \to s^{+}$, to $u_0$ in the $L_2(D)$\nobreakdash-\hspace{0pt}norm (for details see, e.g., \cite[Chapter XVI]{DL}).

Furthermore, for each $t > s$ the linear mapping $\bigl[\, u_0 \mapsto
u(t; \cdot; s, u_0) \,\bigr]$ is a bounded (even completely continuous) linear operator from $L_2(D)$ into $L_2(D)$.  We will denote this operator by $U(t;s)$.

The family $\{ U(t; s) \}_{s \le t}$ satisfies the following equalities
\begin{equation*}
U(t; r) = U(t; s) U(s; r) \quad \text{for any } r \le s \le t,
\end{equation*}
and
\begin{equation*}
U(t + T; s + T) = U(t; s) \quad \text{for any } s \le t.
\end{equation*}

A solution $u \colon (0,\infty) \times \overline{D} \to \RR$ of \eqref{eq:PDE-parabolic} satisfying the initial condition $u(0,\cdot) = u_0$, where $u_0 \in L_2(D)$, is called {\em positive\/} if $u(t,x) > 0$ for all $t > 0$ and all $x \in D$.  It follows from standard maximum principles (see \cite[Chapter 3]{PrWein}) that a necessary and sufficient condition for a solution of \eqref{eq:PDE-parabolic} to be positive is that $u_0(x) \ge 0$ for a.e.\ $x \in D$ and $u_0 \not\equiv 0$.

A solution $u \colon \RR \times \overline{D} \to \RR$ of \eqref{eq:PDE-parabolic} is called {\em globally positive\/} if $u(t,x) > 0$ for all $t \in \RR$ and all $x \in D$.

\section{Principal Lyapunov Exponent for Time-Periodic Parabolic PDEs of Second Order}

The theory of principal Lyapunov exponents for time\nobreakdash-\hspace{0pt}periodic PDEs of second order was developed by P.\ Hess and presented in his monograph~\cite{Hess}.
\begin{theorem}
\label{thm:PDE-Floquet}
For a linear $T$\nobreakdash-\hspace{0pt}periodic problem \eqref{eq:PDE-parabolic} there exist and are uniquely determined: a number $\rho > 0$, and a globally positive solution $v(t,x)$ of \eqref{eq:PDE-parabolic}, such that
\begin{itemize}
\item[{\rm (i)}]
$v(T, x) = {\rho} v(0, x)$ for all $x \in D$;
\item[{\rm (ii)}]
$\norm{v(0, \cdot)} = 1$.
\end{itemize}
Consequently, the {\em logarithmic growth rate\/} of $v(t,x)$, $\lim\limits_{t \to \infty} \frac{1}{t}\ln\norm{v(t,\cdot)}$, exists and is equal to $\frac{1}{T}\ln\rho$.  Moreover,
\begin{itemize}
\item[{\rm (iii)}]
\begin{equation*}
\lim\limits_{t \to \infty} \frac{\ln{\norm{U(t;0)}}}{t} = \frac{\ln{\rho}}{T};
\end{equation*}
\item[{\rm (iv)}]
for any positive solution $u(t,x)$ of \eqref{eq:PDE-parabolic} there holds
\begin{equation*}
\lim\limits_{t \to \infty} \frac{\ln{\norm{u(t, \cdot)}}}{t} =
\frac{\ln{\rho}}{T}.
\end{equation*}
\end{itemize}
\end{theorem}
The proof of the above theorem (except uniqueness of globally positive solutions) was given in~\cite{Hess}.  In it one uses the fact that, by standard parabolic maximum principles, the completely continuous operator $U(T;0)$ satisfies some positivity properties, which allows us to use the Kre\u{\i}n--Rutman theorem (\cite[Theorems 19.2 and 19.3]{Deim}) to conclude that the spectral radius $\rho$ of $U(T;0)$ is a positive simple eigenvalue and that an eigenfunction pertaining to it can be taken to assume positive values on $D$.  As $v(t,x)$ we take the solution of~\eqref{eq:PDE-parabolic} which equals, for $t = 0$, such a normalized eigenfunction.
\begin{remark}
{\em In~\cite{Hess} it was also shown that $\Lambda$ is an eigenvalue of the differential operator (in our notation)
\begin{equation*}
\Bigl[ u \mapsto \frac{\p u}{\p t} - {\Delta} u - c_0(t,x) u \Bigr]
\end{equation*}
acting on some Banach space of H\"older continuous functions defined on $\RR \times \overline{D}$, $T$\nobreakdash-\hspace{0pt}periodic in $t$, and that $e^{{-\Lambda}t} v(t,x)$ is an eigenfunction corresponding to that eigenvalue. We will not pursue that (potentially promising) way of thinking in the present review.}
\end{remark}

\begin{remark}
{\em The fact that $v(t,x)$ is a unique, up~to multiplication by positive scalars, globally positive solution of~\eqref{eq:PDE-parabolic}, is a (not very straightforward) consequence of the parabolic Harnack estimates, see, e.g., \cite{Ni}, \cite{Mi0}, or~\cite{HuPoSa}.

It appears to the author that the uniqueness of globally positive solutions is not needed to obtain the majority of (all?) results mentioned in the present survey, but it greatly simplifies their proofs.}
\end{remark}

The number
\begin{equation*}
\Lambda := \frac{\ln{\rho}}{T}
\end{equation*}
is called the {\em principal Lyapunov exponent\/} (sometimes {\em principal Floquet exponent\/}) of \eqref{eq:PDE-parabolic}.

\medskip
Since $v(t,x)$ is a classical solution, we can prove the following analog of~\eqref{eq:ODE-lambda}:
\begin{equation}
\label{eq:PDE-lambda}
\Lambda = \frac{1}{T} \int\limits_{0}^{T} \frac{\langle (\Delta +
c(t,\cdot)) v(t, \cdot), v(t,\cdot) \rangle}{\norm{v(t, \cdot)}^2} \, dt.
\end{equation}

\section{Principal Lyapunov Exponents for Parabolic PDEs: ``Trivial'' Estimate}
\label{section:PDE-trivial}
For a linear $T$\nobreakdash-\hspace{0pt}periodic problem \eqref{eq:PDE-parabolic}, denote
\begin{equation*}
\begin{aligned}
c_{\mathrm{min}}(x) & := \min\limits_{t \in [0, T]} c(t,x), \quad x \in \overline{D},
\\
c_{\mathrm{max}}(x) & := \max\limits_{t \in [0, T]} c(t,x), \quad x \in \overline{D}.
\end{aligned}
\end{equation*}
\begin{proposition}
\label{prop:PDE-trivial}
For a linear $T$\nobreakdash-\hspace{0pt}periodic problem
\eqref{eq:PDE-parabolic} there holds
\begin{equation*}
\lambda_{\mathrm{princ}}(\Delta + c_{\mathrm{min}}(\cdot) I) \le \Lambda \le
\lambda_{\mathrm{princ}}(\Delta + c_{\mathrm{max}}(\cdot) I).
\end{equation*}
\end{proposition}
The proof of the above result is almost a copy of the proof of Proposition~\ref{prop:ODE-trivial}, with comparison theorems for systems of ordinary differential inequalities replaced by theorems on parabolic partial differential inequalities (see, e.g., \cite[Chapter 3]{PrWein}).

\section{Parabolic PDEs: An Estimate from Above}
\label{section:PDEs-above}
Note that for each $t \in \RR$ integration~by~parts gives that
\begin{equation*}
\left\langle (\Delta + c(\cdot)) v(t, \cdot), v(t, \cdot) \right\rangle = \int\limits_{D} \left( - \norm{{\nabla}_{x} v(t,x)}_{\RR^n}^2  + c(t,x) v^2(t,x) \right) \, dx.
\end{equation*}
By applying the above formula into~\eqref{eq:PDE-lambda} we obtain, via Proposition~\ref{prop:Rayleigh}, the following (folk?) result:
\begin{proposition}
\label{prop:PDE-above}
For a linear $T$\nobreakdash-\hspace{0pt}periodic problem \eqref{eq:PDE-parabolic} there holds
\begin{equation*}
\Lambda \le \frac{1}{T} \int\limits_{0}^{T} \lambda_{\mathrm{princ}}({\Delta} +
c(t, \cdot)I) \, dt.
\end{equation*}
\end{proposition}

\section{Parabolic PDEs: An Estimate from Below}
\label{section:PDEs-below}
Denote
\begin{equation*}
\overline{c}(x) = \frac{1}{T} \int\limits_{0}^{T} c(t, x) \, dt, \quad x
\in \overline{D}.
\end{equation*}
The following result was proved in~\cite{HShV}:
\begin{theorem}
\label{thm:PDE-averaging}
For a linear $T$\nobreakdash-\hspace{0pt}periodic problem \eqref{eq:PDE-parabolic} we have
\begin{equation*}
\Lambda \ge \lambda_{\mathrm{princ}}({\Delta} + \overline{c}(\cdot)I).
\end{equation*}
\end{theorem}
Also, in~\cite{HShV} it was proved that the equality holds if~and~only~if $c(t,x) = c_1(t) + c_2(x)$ for some $c_1$, $c_2$.

\part{Generalizations to Nonperiodic Time Dependence}
\label{part:generalizations-non-periodic}
In the present Part we present generalizations of (some) results in Parts~\ref{part:ODEs} and~\ref{part:PDEs} to the case when the dependence of the matrices $C$ or functions $c$ on time is almost~periodic (in Section~\ref{section:almost-periodic}), general nonautonomous (in Section~\ref{section:general-nonautonomous}), or, finally, random (in Section~\ref{section:random}).

\section{Almost Periodic Case}
\label{section:almost-periodic}
Assume that in
\begin{equation}
\label{eq:ODE-almost-periodic}
\frac{du}{dt} = C(t) u,
\end{equation}
the matrix function $C(\cdot)$ is (Bohr) almost periodic (see, e.g., \cite{Fink}), with the property that there exists $\delta > 0$ such that $c_{ij}(t) \ge \delta$ for all $t \in \RR$, $i \ne j$.  Then Theorem~\ref{thm:Floquet-ODE} has the following counterpart:  There exists a unique globally positive solution $v(t)$ of~\eqref{eq:ODE-almost-periodic} such that $\norm{v(0)} = 1$, and its logarithmic growth rate, $\lim\limits_{t \to \infty} \frac{1}{t}\ln\norm{v(t)}$, exists (denote it by $\Lambda$).  Further, it can be proved that
\begin{itemize}
\item
\begin{equation*}
\lim\limits_{t \to \infty} \frac{\ln{\norm{U(t;0)}}}{t} = \Lambda;
\end{equation*}
\item
for any positive solution $u(t)$ of \eqref{eq:ODE-almost-periodic} there holds
\begin{equation*}
\lim\limits_{t \to \infty} \frac{\ln{\norm{u(t)}}}{t} = \Lambda.
\end{equation*}
\end{itemize}
The main difference between the periodic and almost periodic cases is that in the latter there is no linear operator (like $U(T;0)$ in the former) whose spectral radius is an eigenvalue such that a corresponding eigenvector serves as a starting point of a globally positive solution.  Instead, the construction is much more involved.  However, when we have already constructed such a solution, the proofs of estimates are much the same.

\medskip
A similar situation holds for PDEs
\begin{equation}
\label{eq:PDE-almost-periodic}
\begin{cases}
\displaystyle \frac{\partial u}{\partial t} = {\Delta} u + c(t,x) u, & x \in D,
\\[1ex]
u(t,x) = 0, & x \in \partial D,
\end{cases}
\end{equation}
where $c \colon \RR \times \overline{D} \to \RR$ is a sufficiently regular function, almost periodic in $t$ uniformly in $\overline{D}$ (see~\cite{ShYi}).

\section{General Nonautonomous Case}
\label{section:general-nonautonomous}
In a general case of nonautonomous strongly cooperative systems of linear ODEs:
\begin{equation}
\label{eq:ODE-nonautonomous}
\frac{du}{dt} = C(t) u,
\end{equation}
where $C \colon \RR \to \RR^{N \times N}$ is bounded and (globally uniformly) Lipschitz continuous, with the property that there exists $\delta > 0$ such that $c_{ij}(t) \ge \delta$ for all $t \in \RR$, $i \ne j$, it can be proved that there exists a unique (up~to multiplication by positive scalars) globally positive solution $v(t)$ of~\eqref{eq:ODE-nonautonomous} such that $\norm{v(0)} = 1$.  However, in~general one can prove only that there are reals $\Lambda_{\mathrm{min}} \le \Lambda_{\mathrm{max}}$ such that
\begin{equation*}
\liminf_{t - s \to \infty} \frac{\ln\norm{v(t)} - \ln\norm{v(s)}}{t - s} = \Lambda_{\mathrm{min}} \text{ and } \limsup_{t - s \to \infty} \frac{\ln\norm{v(t)} - \ln\norm{v(s)}}{t - s} = \Lambda_{\mathrm{max}}.
\end{equation*}
The interval $[\Lambda_{\mathrm{min}}, \Lambda_{\mathrm{max}}]$ is called the {\em principal spectrum\/} of~\eqref{eq:ODE-nonautonomous}.

\smallskip
The situation becomes slightly different in the {\em forward nonautonomous\/} case, that is, when the matrix function $C(\cdot)$ is defined and satisfies the corresponding properties on $[0, \infty)$ only.  Namely, there is obviously no globally positive solution of the system, but there are reals $\Lambda_{\mathrm{min}} \le \Lambda_{\mathrm{max}}$ such that
\begin{equation*}
\liminf_{\substack{t - s \to \infty \\ s \to \infty}} \frac{\ln\norm{u(t)} - \ln\norm{u(s)}}{t - s} = \Lambda_{\mathrm{min}} \text{ and } \limsup_{\substack{t - s \to \infty \\ s \to \infty}} \frac{\ln\norm{u(t)} - \ln\norm{u(s)}}{t - s} = \Lambda_{\mathrm{max}}.
\end{equation*}
for \textbf{any} positive solution.  Similarly, $[\Lambda_{\mathrm{min}},  \Lambda_{\mathrm{max}}]$ is called the {\em principal spectrum\/} of the forward nonautonomous~\eqref{eq:ODE-nonautonomous}.

\medskip
For an extension of the above theory to linear parabolic PDEs of second order, see the monograph~\cite{MiSh} (in the nonautonomous case) and the paper~\cite{MiSh-Fields} (in the forward nonautonomous case).

\smallskip
Many of the results mentioned in Parts~\ref{part:ODEs} and~\ref{part:PDEs} carry over to the nonautonomous case and/or forward nonautonomous case, with suitable modifications.

\section{Random Case}
\label{section:random}
Random systems form another generalization of time\nobreakdash-\hspace{0pt}periodic systems.  To give a flavor of those generalizations we briefly describe, without going into details, random systems of linear strongly cooperative ODEs.

Assume that $(\Omega, \mathfrak{F}, \PP)$ is a measure space: $\Omega$ is a set, $\mathfrak{F}$ is a $\sigma$\nobreakdash-\hspace{0pt}algebra of subsets of $\Omega$ and $\PP$ is a probability measure on $\mathfrak{F}$.  We always assume that $\PP$ is complete.

We let $\theta = (\theta_t)_{t \in \RR}$ be a $\PP$\nobreakdash-\hspace{0pt}preserving ergodic dynamical system on a measure space $(\Omega, \mathfrak{F}, \PP)$.  For $\omega \in \Omega$ and $t \in \RR$ we usually write $\omega \cdot t$ instead~of $\theta_{t}\omega$.

\smallskip
Consider a family of linear systems of ODEs
\begin{equation}
\label{eq:ODE-random}
\frac{du}{dt} = C(\omega \cdot t) u
\end{equation}
indexed by $\omega \in \Omega$.  Here $C \colon \Omega \to \RR^{N \times N}$ is measurable.  Sometimes we will write \eqref{eq:ODE-random}$_{\omega}$ to emphasize that we consider the system for a fixed $\omega \in \Omega$.

\smallskip
We present now, in a rather cursory way, the main points of the theory of random strongly cooperative systems of linear ODEs as developed in~\cite{MiSh2}.  That theory is a specialization of an abstract theory of random monotone dynamical systems presented in~\cite{MiSh1}.

Firstly, one assumes that for each $\omega \in \Omega$ the function $\bigl[\, \RR \ni t \mapsto C(\omega \cdot t) \in \RR^{N \times N} \,\bigr]$ is continuous.  This guarantees the existence of solutions of~\eqref{eq:ODE-random}$_{\omega}$.  Let $U_{\omega}(t;s)$ stand for the corresponding transition matrix.

Secondly, one assumes that the matrix function $C \colon \Omega \to \RR^{N \times N}$ belongs to $L_{1}\bigl((\Omega, \mathcal{F}, \PP), \RR^{N \times N}\bigr)$.  That assumption guarantees that \eqref{eq:ODE-random} gives rise to a random (skew\nobreakdash-\hspace{0pt}product) dynamical system to which ergodic theorems can be applied (see \cite[Chapter 2, in particular Example 2.2.8]{Arn}).

Regarding strong cooperativity, the condition alone that $C(\omega) \in \mathcal{M}_N$  does not suffice to reproduce many features of the theory of principal Lyapunov exponents for time\nobreakdash-\hspace{0pt}periodic systems.   In~\cite{MiSh2} sufficient conditions are given in the form of some functions, expressed in~terms of the entries of $C(\omega)$, belonging to $L_1\bigl(\OFP, \RR\bigr)$.  For~instance, those conditions are satisfied if the matrices $C(\omega)$ are bounded uniformly in $\omega \in \Omega$ and their off\nobreakdash-\hspace{0pt}diagonal entries are bounded away from zero uniformly in $\omega \in \Omega$.

It is proved then that there exist and are uniquely determined: a real number $\Lambda$ and a family $\{v_{\omega}(t)\}$ of globally positive solutions of~\eqref{eq:ODE-random}$_{\omega}$, indexed by $\omega$ belonging to some set $\Omega' \subset \Omega$ of full measure, with $v_{\omega}(0)$ depending measurably on $\omega \in \Omega'$, satisfying, for each $\omega \in \Omega'$, the following:
\begin{itemize}
\item
$\norm{v_{\omega}(0)} = 1$;
\item
for each $t \in \RR$, $v_{\omega}(t)$ equals $v_{\omega \cdot t}(0)$ multiplied by some positive number (dependent perhaps on $\omega$ and $t$);
\item
\begin{equation*}
\Lambda = \lim\limits_{t \to \infty} \frac{\ln\norm{v_{\omega}(t)}}{t}.
\end{equation*}
\end{itemize}
Further, it is proved that for any $\omega \in \Omega'$,
\begin{equation*}
\lim\limits_{t \to \infty} \frac{\ln\norm{U_{\omega}(t;0)}}{t} = \Lambda,
\end{equation*}
and that for any positive solution of~\eqref{eq:ODE-random}$_{\omega}$, $\omega \in \Omega'$,
\begin{equation*}
\lim\limits_{t \to \infty} \frac{\ln\norm{u(t)}}{t} = \Lambda.
\end{equation*}
The number $\Lambda$ is called the {\em principal Lyapunov exponent\/} of~\eqref{eq:ODE-random}.

Eq.~\eqref{eq:ODE-lambda} has the following analog:
\begin{equation}
\label{eq:ODE-Lambda-AP}
\Lambda = \lim\limits_{t \to \infty }\frac{1}{t} \int\limits_{0}^{t}
\langle C(\omega) v_{\omega}(\tau), v_{\omega}(\tau) \rangle \, d\tau, \quad \omega \in \Omega'.
\end{equation}
An application of Birkhoff's ergodic theorem (see~\cite[Appendix A]{Arn}) allows us to write
\begin{equation}
\label{eq:ODE-Lambda-AP-2}
\Lambda = \int\limits_{\Omega} \langle C(\omega) v_{\omega}(0), v_{\omega}(0) \rangle \, d\PP(\omega).
\end{equation}
Applying~\eqref{eq:ODE-symmetrisation} we obtain the following generalization of~Proposition~\ref{prop:ODEs-above}:
\begin{equation}
\label{eq:ineq-AP}
\Lambda \le \int\limits_{\Omega} \lambda_{\mathrm{dom}}(C^{\mathrm{S}}(\omega)) \, d\PP(\omega).
\end{equation}

We have the following counterpart of Theorem~\ref{thm:Kolotilina-Frobenius} (see~\cite{Mi}):
\begin{gather}
\Lambda \ge \frac{1}{N} \int\limits_{\Omega} \Bigl(
\trace{C(\omega)} + 2 \sum\limits_{j < k} \sqrt{c_{jk}(\omega)
c_{kj}(\omega)} \; \Bigr) \, d\PP(\omega),
\\
\Lambda \ge \int\limits_{\Omega} \Bigl( \min\limits_{i}
\sum\limits_{j} c_{ij}(\omega) \Bigr) \, d\PP(\omega),
\\
\Lambda \ge \int\limits_{\Omega} \Bigl( \min\limits_{j}
\sum\limits_{j} c_{ij}(\omega) \Bigr) \, d\PP(\omega).
\end{gather}

When the off\nobreakdash-\hspace{0pt}diagonal entries of $C(\omega)$ are independent of~$\omega$ we have the following generalization of Theorem~\ref{thm:averaging-ODEs}:
\begin{equation*}
\Lambda \ge \lambda_{\mathrm{dom}}(\overline{C}),
\end{equation*}
where $\overline{C} = [\overline{c}_{ij}]_{i,j = 1}^{N}$ is the matrix with $\overline{c}_{ij} = c_{ij}$ for $i \ne j$ and
\begin{equation*}
\overline{c}_{ii} = \int\limits_{\Omega} c_{ii}(\omega) \, d\PP(\omega).
\end{equation*}
In other words, the principal Lyapunov exponent of the original system is not smaller than the dominant eigenvalue of the state\nobreakdash-\hspace{0pt}averaged system.

A proof is a copy of the proof of Theorem~\ref{thm:averaging-ODEs}, where we take the globally positive solution $v_{\omega}(t)$ for some $\omega \in \Omega'$, replace $(1/T) \int_{0}^{T} \dots$ with $\lim_{t\to\infty}(1/t) \int_{0}^{t} \dots$, and apply Birkhoff's ergodic theorem.

\subsection{Application to Investigation of Almost Periodic Systems}
\label{subsection:random-AP}
Ideas from the random case can be fruitfully applied to investigating deterministic cases.

\medskip
For example, in the almost periodic case the completion of the set of all time\nobreakdash-\hspace{0pt}translates of $C(\cdot)$ (for~\eqref{eq:ODE-almost-periodic}) in an appropriately chosen metric is a compact metrizable space (the \textit{hull} $\Omega$ of $C(\cdot)$) on which the translation flow is minimal and, moreover, uniquely ergodic (meaning that there exists a unique probability Borel measure $\PP$ on $\Omega$ that is invariant with respect to the translation flow).  For details, see \cite{ShYi}.

In~\cite[Example 4.2]{COS} numerical estimates of $\Lambda$ are given for a two-dimensional system $du/dt = C(t)u$, quasi\nobreakdash-\hspace{0pt}periodic in~time.  There, an approximation of $v_{\omega}(0)$ is obtained, and formula~\ref{eq:ODE-Lambda-AP-2} is applied.

\part{Generalizations to Other Differential Equations}
\label{part:generalization-other-DE}

\section{Cooperative Systems of ODEs}
\label{sect:ODE-cooperative}
System
\begin{equation*}
\frac{du}{dt} = C(t) u,
\end{equation*}
where $C \colon \RR \to \RR^{N \times N}$ is continuous and $T$\nobreakdash-\hspace{0pt}periodic, is called {\em cooperative\/} if for each $t \in \RR$ the matrix $C(t)$ has all off\nobreakdash-\hspace{0pt}diagonal entries nonnegative. A theorem due to M\"uller and Kamke (see, e.g., \cite[Chapter II]{W}, or \cite[Chapter 3]{HiSm}) states that for a cooperative system the transition matrix $U(t;s)$ is nonnegative.

\medskip
For autonomous linear cooperative systems a simple condition is known, called {\em irreducibility\/} (for definition, see, e.g., \cite[Definition~(1.2)]{BP}), which is equivalent to $e^{tC}$ being positive for all $t > 0$.  The ``sufficiency'' part of that condition can be generalized to the case of time\nobreakdash-\hspace{0pt}periodic systems (and even to random systems, see Assumption~(O3) in~\cite{MiSh2}), and then the whole theory presented in Part~\ref{part:ODEs} carries over word~for~word.

\medskip
Even in the case of systems that are only cooperative something can be said.  There need not be a (unique up~to multiplication by positive scalars) globally positive solution $v(t)$ (for that matter, there need not be any globally positive solution), nonetheless the (necessarily positive) spectral radius of $U(T;0)$ is an eigenvalue (not necessarily simple), to which there exists a nonnegative eigenvector.

One can in a natural way define the {\em top Lyapunov exponent\/} $\Lambda$ as the logarithmic growth rate of the norm of $U(t; 0)$:
\begin{equation*}
\Lambda = \lim\limits_{t \to \infty} \frac{\ln{\norm{U(t; 0)}}}{t}.
\end{equation*}
In~particular, some estimates in~\cite{Mi}, for~example estimates as in Theorem~\ref{thm:Kolotilina-Frobenius}, hold, after suitable modifications, for cooperative systems.

\section{Strongly Cooperative Systems of \\ Parabolic PDEs}
\label{sect:PDE-systems}
Consider a system of linear reaction--diffusion equations
\begin{equation}
\label{eq:PDE-system}
\begin{cases}
\displaystyle \frac{\partial u_{i}}{\partial t} = {\Delta} u_{i} +
\sum\limits_{j=1}^{n} c_{ij}(t,x) u_j, & \quad 1 \le i \le N, \ x \in \Omega,
\\[1ex]
u_i(t,x) = 0, & \quad 1 \le i \le N, \ x \in \p D,
\end{cases}
\end{equation}
where the $C^3$ functions $c_{ij} \colon \RR \times \overline{D} \to \RR$ are assumed to be time\nobreakdash-\hspace{0pt}periodic with period $T$.

System~\eqref{eq:PDE-system} is called {\em strongly cooperative\/} if for each $t \in \RR$ and $x \in \overline{D}$ the matrix $[c_{ij}(t,x)]_{i,j=1}^{N}$ is in $\mathcal{M}_N$.

\smallskip
For strongly cooperative systems~\eqref{eq:PDE-system} an analog of Theorem~\ref{thm:PDE-Floquet} holds, where by a {\em globally positive solution\/} of~\eqref{eq:PDE-system} one understands a solution  $v \colon \RR \times \overline{D} \to \RR^N$ such that $v_i(t, x) > 0$ for all $1 \le i \le N$, $t \in \RR$ and $x \in D$.

\medskip
A theory of the principal spectrum for autonomous (and the principal Lyapunov exponent for random) strongly cooperative systems of linear reaction--diffusion equations was presented in~\cite{MiSh}.  In~particular, for systems~\eqref{eq:PDE-system} such that $c_{ij}$ does not, for $i \ne j$, depend on time an analog of Theorem~\ref{thm:PDE-averaging} was proved (see~\cite[Subsection 6.3.3]{MiSh}).

\section{More General Linear Parabolic PDEs of Second Order}
\label{sect:PDEs-general}

\subsection{General Linear Parabolic PDEs}
\label{subsect:PDE-general}
Consider a Dirichlet problem for a linear parabolic partial differential equation of second order
\begin{equation}
\label{eq:PDE-most-general}
\begin{cases}
\displaystyle
\frac{\p u}{\p t} = \displaystyle \sum_{i,j=1}^{n} a_{ij}(t,x) \frac{\p^2 u}{\p x_{i} x_{j}} + \sum_{i=1}^n b_i(t,x) \frac{\p u}{\p x_i} + c(t,x)u, & \quad  x \in D,
\\[1.5ex]
u(t, x)  = 0, & \quad x \in \p D
\end{cases}
\end{equation}
on a bounded domain $D \subset \RR^n$ with sufficiently regular boundary $\p D$, where the (suitably regular) coefficients $a_{ij}$, $a_i$, $b_i$ and $c$ are time\nobreakdash-\hspace{0pt}periodic with period $T$, $a_{ij} \equiv a_{ji}$ for any $i \ne j$, and there is $\gamma > 0$ such that
\begin{equation*}
\sum\limits_{i,j=1}^{n} a_{ij}(t, x) \xi_{i} \xi_{j} \ge \gamma \vert\vert \xi \vert\vert_{\RR^n}^{2}, \quad \forall \, t \in \RR, \ \forall \, x \in D, \ \forall \, \xi = (\xi_1, \dots, \xi_n) \in \RR^n.
\end{equation*}
For such systems, it was proved in~\cite{Hess} that there are a number $\rho > 0$ and a globally positive solution $v(t,x)$ of~\eqref{eq:PDE-most-general} such that all the properties as in Theorem~\ref{thm:PDE-Floquet} are satisfied.

When the second- and first\nobreakdash-\hspace{0pt}order coefficients $a_{ij}$ and $b_i$ are independent of time, an analog of the ``trivial'' estimates in Proposition~\ref{prop:PDE-trivial} holds for problems~\eqref{eq:PDE-most-general}:
\begin{equation*}
\lambda_{\mathrm{princ}}^{\mathrm{min}} \le \Lambda \le \lambda_{\mathrm{princ}}^{\mathrm{max}},
\end{equation*}
where $\lambda_{\mathrm{princ}}^{\mathrm{min}}$ stands for the principal eigenvalue of the elliptic boundary value problem
\begin{equation}
\label{eq:PDE-most-general-elliptic}
\begin{cases}
\displaystyle \sum_{i,j=1}^{n} a_{ij}(x) \frac{\p^2 u}{\p x_{i} x_{j}} + \sum_{i=1}^n b_i(x) \frac{\p u}{\p x_i} + \Bigl( \min\limits_{t \in [0,T]} c(t,x) \Bigr) u = 0, & \quad  x \in D,
\\[1.5ex]
u(x)  = 0, & \quad x \in \p D
\end{cases}
\end{equation}
(and similarly for $\lambda_{\mathrm{princ}}^{\mathrm{max}}$).

\medskip
The Dirichlet boundary condition in~\eqref{eq:PDE-most-general} can be replaced by the Neumann boundary condition
\begin{equation*}
\sum\limits_{i=1}^{n} \frac{\p u}{\p x_{i}} \nu_i(t,x) = 0, \quad x \in \p D,
\end{equation*}
or by the Robin boundary conditions
\begin{equation*}
\sum\limits_{i=1}^{n} \frac{\p u}{\p x_{i}} \nu_i(t,x) + d(t,x) u = 0, \quad x \in \p D,
\end{equation*}
where $\nu = (\nu_1, \dots, \nu_n) \colon \RR \times \p D \to \RR^n$ is a $T$\nobreakdash-\hspace{0pt}periodic (in $t$) sufficiently regular vector field on $\p D$ pointing out of $D$ and $d \colon \RR \times \p D \to [0,\infty)$ is $T$\nobreakdash-\hspace{0pt}periodic (in $t$) sufficiently regular function.  Indeed, other boundary conditions are allowed as long as they admit the strong maximum principle for the corresponding problem.

\subsection{Linear Parabolic PDEs in Divergence Form}
\label{subsect:PDE-divergence-form}
Let us specialize to Dirichlet problems for linear parabolic PDEs of second order in {\em divergence form\/}:
\begin{equation}
\label{PDE}
\left\{
\begin{aligned}
\displaystyle
& \frac{\p u}{\p t} = \displaystyle \sum_{i=1}^{n} \frac{\p}{\p x_i}  \biggl( \sum_{j=1}^{n} a_{ij}(t,x) \frac{\p u}{\p x_{j}} + a_{i}(t,x) u \biggr)
\\
& \quad \quad + \sum_{i=1}^n b_i(t,x)\frac{\p u}{\p x_i} + c(t,x)u, & \quad  x \in D,
\\[1ex]
& u(t, x) = 0, & \quad  x \in \p D.
\end{aligned}
\right.
\end{equation}
After performing some integration by parts a reader who is well versed in the existence theory of parabolic PDEs will notice that, for~\eqref{PDE},
\begin{equation*}
\Lambda = - \frac{1}{T} \int\limits_{0}^{T} B(t; w(t), w(t)) \, dt,
\end{equation*}
where $B(t; \cdot, \cdot)$ denotes the {\em Dirichlet form\/}:
\begin{equation}
\label{def:Dirichlet-form}
\begin{aligned}
B(t; u, v) & := \int\limits_{D} \biggl( \sum_{i=1}^{n} \Bigl(
\sum_{j=1}^{n} a_{ij}(t,x) \frac{\p u}{\p x_{j}} + a_{i}(t,x) \Bigr)
\frac{\p v}{\p x_{i}} \biggr) \, dx \\
& - \int\limits_{D} \Bigl( \sum_{i=1}^{n} b_{i}(t,x) \frac{\p u}{\p
x_{i}} + c(t,x) \Bigr) v \, dx, \quad u, v \in H_{0}^{1}(D).
\end{aligned}
\end{equation}
If $a_i$ and $b_i$ are constantly equal to zero, that is, in the case of
\begin{equation}
\label{eq:PDE-self-adj}
\begin{cases}
\displaystyle
\frac{\p u}{\p t} = \displaystyle \sum_{i=1}^{n} \frac{\p}{\p x_i}  \biggl( \sum_{j=1}^{n} a_{ij}(t,x) \frac{\p u}{\p x_{j}} \biggr) + c(t,x)u, & x \in D,
\\[1ex]
u(t, x)  = 0, & x \in \p D,
\end{cases}
\end{equation}
the corresponding differential operator is, for all $t$, self\nobreakdash-\hspace{0pt}adjoint, and the Dirichlet form is symmetric:
\begin{equation*}
B(t; u, v) = \int\limits_{D} \biggl( \sum_{i,j=1}^{n} a_{ij}(t,x) \frac{\p u}{\p x_{j}} \frac{\p v}{\p x_{i}} + c(t,x) u v \biggr) \, dx, \quad u, v \in H_{0}^{1}(D).
\end{equation*}
Then the analog of Proposition~\ref{prop:PDE-above} holds:
\begin{equation*}
\Lambda \le \frac{1}{T} \int\limits_{0}^{T} \lambda_{\mathrm{princ}}(A(t) + c(t, \cdot) I) \, dt,
\end{equation*}
where $\lambda_{\mathrm{princ}}(A(t) + c(t, \cdot) I)$ denotes the principal eigenvalue of the elliptic boundary value problem
\begin{equation}
\label{eq:PDE-self-adj-elliptic}
\begin{cases}
\displaystyle
\displaystyle \sum_{i=1}^{n} \frac{\p}{\p x_i}  \biggl( \sum_{j=1}^{n} a_{ij}(t,x) \frac{\p u}{\p x_{j}} \biggr) + c(t,x)u = 0, & x \in D,
\\[1ex]
u(x)  = 0, & x \in \p D.
\end{cases}
\end{equation}

\medskip
The Dirichlet boundary condition in~\eqref{PDE} can be replaced by the Neumann boundary condition
\begin{equation*}
\sum_{i=1}^N \Bigl( \sum_{j=1}^N a_{ij}(t, x) \frac{\p u}{\p  x_j} + a_{i}(t, x) u \Bigr) \nu_i = 0, \quad x \in \p D,
\end{equation*}
or the Robin boundary conditions
\begin{equation*}
\sum_{i=1}^N \Bigl( \sum_{j=1}^N a_{ij}(t, x) \frac{\p u}{\p x_j} + a_{i}(t, x) u \Bigr) \nu_i + d(t,x) u = 0, \quad x \in \p D,
\end{equation*}
where now $\nu = (\nu_1, \dots, \nu_n)$ denotes the unit normal vector field on $\p D$
pointing out~of $D$ and $d \colon \RR \times \p D \to [0,\infty)$ is $T$\nobreakdash-\hspace{0pt}periodic (in $t$) sufficiently regular function.

\subsection{Faber--Krahn Inequalities}
\label{subsect:Faber-Krahn}
In the case of
\begin{equation}
\label{eq:PDE-Faber-Krahn}
\begin{cases}
\displaystyle
\frac{\p u}{\p t} = \sum_{i=1}^{n} \frac{\p}{\p x_i}  \biggl( \sum_{j=1}^{n} a_{ij}(t,x) \frac{\p u}{\p x_{j}} \biggr), & \quad x \in D,
\\[1ex]
u(t, x)  = 0, & \quad x \in \p D,
\end{cases}
\end{equation}
let $\gamma \colon \RR \to (0,\infty)$ be a continuous $T$\nobreakdash-\hspace{0pt}periodic function such that for each $t \in \RR$,
\begin{equation*}
\sum\limits_{i,j=1}^{n} a_{ij}(t,x) \xi_{i} \xi_{j} \ge \gamma(t) \vert\vert \xi \vert\vert_{\RR^n}^2, \quad \forall \, x \in D, \ \forall \, \xi = (\xi_1, \dots, \xi_n) \in \RR^n.
\end{equation*}
Then
\begin{equation*}
\Lambda \le \Bigl( \frac{1}{T} \int\limits_{0}^{T} \gamma(t) \, dt \Bigr) \lambda_{\mathrm{princ}}^{\mathrm{sym}},
\end{equation*}
where $\lambda_{\mathrm{princ}}^{\mathrm{sym}}$ denotes the principal eigenvalue of the Dirichlet Laplacian on the ball in $\RR^n$ whose volume equals the volume of $D$ (for a proof, see \cite{Hess}; see also extensions to nonautonomous and random cases in~\cite{MiSh0} and~\cite[Section 5.4]{MiSh}).

\subsection{Spatial Averaging}
\label{subsect:PDE-spatial-averaging}
For time\nobreakdash-\hspace{0pt}periodic linear parabolic PDEs with Neumann boundary conditions
\begin{equation}
\label{eq:PDE-spatial-averaging}
\begin{cases}
\displaystyle
\frac{\p u}{\p t} = \sum_{i=1}^{n} \frac{\p}{\p x_i}  \biggl( \sum_{j=1}^{n} a_{ij}(t) \frac{\p u}{\p x_{j}} \biggr) + c(t,x) u, & \quad x \in D,
\\[1ex]
\displaystyle \sum_{i=1}^N \Bigl( \sum_{j=1}^N a_{ij}(t) \frac{\p u}{\p x_j}\Bigr) \nu_i = 0, & \quad x \in \p D,
\end{cases}
\end{equation}
where $\nu = (\nu_1, \dots, \nu_n)$ denotes the unit normal vector field on $\p D$
pointing out~of $D$, the principal Lyapunov exponent of~\eqref{eq:PDE-spatial-averaging} is bounded \textbf{below} by the principal Lyapunov exponent of the problem
\begin{equation}
\label{eq:PDE-space-averaged}
\begin{cases}
\displaystyle
\frac{\p u}{\p t} = \sum_{i=1}^{n} \frac{\p}{\p x_i}  \biggl( \sum_{j=1}^{n} a_{ij}(t) \frac{\p u}{\p x_{j}} \biggr) + \widetilde{c}(t) u, & \quad x \in D,
\\[1ex]
\displaystyle \sum_{i=1}^N \Bigl( \sum_{j=1}^N a_{ij}(t) \frac{\p u}{\p x_j}\Bigr) \nu_i = 0, & \quad x \in \p D,
\end{cases}
\end{equation}
where $\widetilde{c}$ is the {\em space average\/} of $c$:
\begin{equation*}
\widetilde{c}(t) := \frac{1}{\abs{D}} \int\limits_{D} c(t,x) \, dx, \qquad t \in [0, T]
\end{equation*}
($\abs{D}$ denotes the volume of $D$).  See \cite[Section 5.3]{MiSh}.  Also, when the coefficients are sufficiently smooth it can be proved that those principal Lyapunov exponents are equal if~and~only~if $c$ is independent of $x$.

\section{Nonlocal Dispersal Equations}
\label{sect:dispersal}
Another generalization of~\eqref{eq:PDE-parabolic} rests on replacing the Laplace operator $\Delta$ (corresponding to the infinitesimal dispersal) with some nonlocal dispersal operator.  To be more specific, consider
\begin{equation}
\label{eq:dispersal}
\frac{\partial u}{\partial t} = \int\limits_{D} K(x - y) \bigl( u(t,y) - u(t,x) \bigr) \, dx + c(t,x) u, \quad x \in D,
\end{equation}
where $K \colon \RR^n \to \RR$ is a suitable kernel and $c \colon \RR \times \overline{D} \to \RR$ is a sufficiently regular function, time\nobreakdash-\hspace{0pt}periodic in $t$ with period $T$.

In the case of nonautonomous nonlocal dispersal equations various concepts related to the principal Lyapunov exponent were investigated in, e.g., \cite{ShV}, \cite{HShV2}, \cite{RSh}, \cite{ShXie}.

It is worth mentioning that frequently the top Lyapunov exponent introduced in the above papers has the property that it is bounded below by the principal eigenvalue of the time\nobreakdash-\hspace{0pt}averaged problem.

\section{(Strongly) Cooperative Delay Differential Equations}
\label{sect:delay}
Consider a linear system of delay differential equations
\begin{equation}
\label{eq:delay}
\frac{d u}{d t}(t) = C(t) u(t) + D(t) u(t-T),
\end{equation}
where $C, D \colon \RR \to \RR^{N \times N}$ are continuous and
$T$\nobreakdash-\hspace{0pt}periodic matrix functions.

Let $\vert\vert \cdot \vert\vert_{\RR^N}$ denote the Euclidean norm on $\RR^N$.

$C([-T,0], \RR^N)$ stands for the Banach space of $\RR^N$\nobreakdash-\hspace{0pt}valued continuous functions defined on $[-T,0]$, with the supremum norm
\begin{equation*}
\norm{u} = \sup\{\, \vert\vert u(\tau) \vert\vert_{\RR^N}: \tau \in [-T, 0] \, \}.
\end{equation*}

It can be proved in the framework of the theory of nonhomogeneous linear systems of ODEs (with no delay) that for any initial time $s \in \RR$ and any initial value $u_0 \in C([-T,0], \RR^N)$
there exists a unique solution $u(\cdot; s, u_0)$ of the initial\nobreakdash-\hspace{0pt}value problem
\begin{equation*}
\begin{cases}
\displaystyle \frac{d u}{d t}(t) = C(t) u(t) + D(t) u(t-T), & t \ge s
\\[1.5ex]
u(t; s, u_0) = u_0(t), & s - T \le t \le s.
\end{cases}
\end{equation*}
Moreover, a little amount of functional analysis allows us to conclude that, if we define a linear operator $U(t; s)$, $s \le t$, from $C([-T, 0], \RR^N)$ into itself by the formula
\begin{equation*}
(U(t;s)u_0)(\tau) := u(t + \tau;s, u_0), \quad \tau \in [-T, 0],
\end{equation*}
the family $\{ U(t; s) \}_{s \le t}$ satisfies
\begin{equation*}
U(t; r) = U(t; s) U(s; r) \quad \text{for any } r \le s \le t,
\end{equation*}
and
\begin{equation*}
U(t + T; s + T) = U(t; s) \quad \text{for any } s \le t.
\end{equation*}
Moreover, the linear operator $U(t; s)$ is completely continuous for $t \ge s + T$.

When we assume that, for any $t \in \RR$, $C(t)$ is in $\mathcal{M}_N$ and $D(t)$ has nonnegative entries (such systems can be called {\em strongly cooperative\/}), we can apply the Kre\u{\i}n--Rutman theorem to $U(T; 0)$ to conclude that there is a unique $w = (w_1, \dots, w_N) \in C([-T,0], \RR^N)$ such that $\norm{w} = 1$, $w_i(\tau) > 0$ for all $1 \le i \le N$ and all $\tau \in [-T,0]$ and $U(T;0) w = {\rho} w$, where $\rho$ denotes the spectral radius of $U(T;0)$.

We can define the {\em principal Lyapunov exponent\/} of a strongly cooperative~\eqref{eq:delay} as
\begin{equation*}
\Lambda = \frac{\ln{\rho}}{T}.
\end{equation*}

Usually it is better to look at $w$ in a different way:  there is a solution $v = (v_1, \dots, v_N) \colon \RR \to \RR^{N}$ of~\eqref{eq:delay} such that for each $t \in \RR$ there holds $v_i(t) > 0$, $1 \le i \le N$, and $\sup\{\, \vert\vert v(\tau) \vert\vert_{\RR^N} : \tau \in [-T,0] \,\} = 1$.  After some effort one can prove that
\begin{equation*}
\lim\limits_{t \to \infty} \frac{\ln{\vert\vert v(t) \vert\vert_{\RR^N}}}{t} = \Lambda,
\end{equation*}
and, moreover, that for any nontrivial solution $u(\cdot) = u(\cdot; s, u_0)$, where $(u_0)_i(s - \tau) \ge 0$ for all $1 \le i \le N$ and all $\tau \in [-T, 0]$ there holds
\begin{equation*}
\lim\limits_{t \to \infty} \frac{\ln{\vert\vert u(t) \vert\vert_{\RR^N}}}{t} = \Lambda,
\end{equation*}

With the help of standard comparison principles for systems of ODEs one can prove analogs of ``trivial'' estimates on $\Lambda$ as in Proposition~\ref{prop:ODE-trivial}.

\medskip
In~\cite{COS}, based on a theory developed in~\cite{NoObS}, the existence of $\Lambda$ satisfying some of the above properties was proved under relaxed assumptions of cooperativity (and with quasi-periodic time dependence).  Furthermore, estimates of $v_{\omega}(0)$ together with a modification of formula~\ref{eq:ODE-Lambda-AP-2} were used to calculate an approximate value of $\Lambda$ (see \cite[Example 5.2]{COS}).

\part{Discrete-Time Counterparts}
\label{part:discrete-time}
As in Section~\ref{section:random}, let $(\Omega, \mathfrak{F}, \PP)$ be a measure space: $\Omega$ is a set, $\mathfrak{F}$ is a $\sigma$\nobreakdash-\hspace{0pt}algebra of subsets of $\Omega$ and $\PP$ is a probability measure on $\mathfrak{F}$.

We let $\theta$ be a $\PP$\nobreakdash-\hspace{0pt}preserving ergodic automorphism of the measure space $(\Omega, \mathfrak{F}, \PP)$.

$\norm{\cdot}$ denotes the Euclidean norm of a vector in $\RR^N$ or a matrix in $\RR^{N \times N}$, depending on the context.

\smallskip
Assume that $M \colon \Omega \to \RR^{N \times N}$ be a measurable mapping.  For $\omega \in \Omega$ and $n \in \NN$ we write
\begin{equation*}
M^{(n)}(\omega) := M(\theta^{n-1}\omega) \cdot \ldots \cdot
M(\omega).
\end{equation*}
Further, let $M^{(0)}(\omega) := I$.

\medskip
{\em Autonomous\/} systems of linear ODEs correspond to iterates, $M^n$, of a matrix $M$.  In that context, strong cooperativity for systems of linear ODEs corresponds to the positivity of $M$.

However, in the random discrete\nobreakdash-\hspace{0pt}time case one needs to impose, besides the nonnegativity of $M(\omega)$, some additional assumptions to guarantee that the theory of principal Lyapunov exponents carry over to that case.  Those assumptions,  given in the form of some functions, expressed in~terms of the entries of $M(\omega)$, belonging to $L_1\bigl(\OFP, \RR\bigr)$ (see \cite{ArnGundDem} and~\cite{MiSh2}), imply that there exist and are uniquely determined: a real number $\Lambda$ and a family $\{w(\omega)\}$ of positive unit vectors, indexed by $\omega$ belonging to some set $\Omega' \subset \Omega$ of full measure, depending measurably on $\omega \in \Omega'$, satisfying, for each $\omega \in \Omega'$, the following:
\begin{itemize}
\item
$M(\omega) w(\omega)$ equals $w(\theta \omega)$ multiplied by some positive number $\mu(\omega)$.
\item
\begin{equation*}
\Lambda = \lim\limits_{n \to \infty} \frac{\ln\norm{M^{(n)}(\omega) w(\omega)}}{n}.
\end{equation*}
\end{itemize}
Further, it is proved that for any $\omega \in \Omega'$,
\begin{equation*}
\lim\limits_{n \to \infty} \frac{\ln\norm{M^{(n)}}}{n} = \Lambda,
\end{equation*}
and that for any nonzero $u \in \RR^{N}_{+}$ and any $\omega \in \Omega'$,
\begin{equation*}
\lim\limits_{t \to \infty} \frac{\ln\norm{M^{(n)}(\omega) u}}{n} = \Lambda.
\end{equation*}
The number $\Lambda$ is called the {\em principal Lyapunov exponent\/} of the random system of positive matrices $\{M(\omega)\}_{\omega \in \Omega}$.

\medskip
It appears that results on estimates of principal Lyapunov exponent in the discrete\nobreakdash-\hspace{0pt}time case are much more numerous than in the continuous\nobreakdash-\hspace{0pt}time case.  For~example, in the case of some special systems explicit formulas for principal Lyapunov exponents are known (see~\cite{Roer} or~\cite{Tul}).

Of special interest is the paper~\cite{ProtJung}, in which homogeneous (of degree one) positive functionals on $\RR^{N}_{+}$ were used to obtain upper and lower estimates of principal Lyapunov exponents (compare the use of homogeneous functionals to prove lower estimates of principal Lyapunov exponents for strongly cooperative linear systems of ODEs in~\cite{Mi}, as described in Subsection~\ref{subsection:ODE-Lyapunov-function}).

In~\cite{BenSch} with the help of the homogeneous polynomial $V(u) = u_1 \cdot \ldots \cdot u_N$ the following lower estimate was obtained:
\begin{theorem}
\label{thm:Benaim-Schreiber}
\begin{equation*}
\Lambda \ge \frac{1}{N} \int\limits_{\Omega} \perm(M(\omega)) \,
d\PP(\omega),
\end{equation*}
where $\perm$ denotes the permanent of a matrix.
\end{theorem}
\begin{problem}
\label{pr:2}
{\em \textbf{Develop a unified approach, via homogeneous functionals, to lower estimates of the principal Lyapunov exponent which could cover both cooperative systems of linear ODEs and random systems of positive matrices.}}
\end{problem}

\subsection{State-Averaging}
\label{subsect:state-averaging}
One could ask if for (some) random systems of positive matrices there is an analog of Theorem~\ref{thm:averaging-ODEs}.

\smallskip
The answer is yes.  Indeed, assume that $M(\omega) = A B(\omega)$, where $A$ is a positive matrix and $B(\omega) = \diag(b_1(\omega), \dots, b_N(\omega))$ with $b_i(\omega) > 0$, $1 \le i \le N$, $\omega \in \Omega$ (this is a natural assumption from the point of view of biological applications, see~\cite{Schreib}).  Further, we make some technical assumptions concerning the summability of some functions expressed in terms of logarithms, maximums and minimums of $b_i(\cdot)$ (those assumptions guarantee that it will be possible to apply Birkhoff's ergodic theorem).

For $1 \le i \le N$ put
\begin{equation*}
\ln{\overline{b}_i} = \int\limits_{\Omega} \ln{b_i(\omega)} \, d\PP(\omega).
\end{equation*}
We claim that the principal Lyapunov exponent $\Lambda$ of the original system is bounded below by the logarithm of the dominant eigenvalue of $A \overline{B}$, where $\overline{B} = \diag(\overline{b}_1, \dots, \overline{b}_N)$.

We write $w(\cdot) = (w_1(\cdot), \dots, w_N(\cdot))$.  Fix $\omega \in \Omega'$ such that
\begin{equation*}
\lim\limits_{n \to \infty} \frac{1}{n}\sum\limits_{k=0}^{n-1}
\ln{w_i(\theta^{k} \omega)} = \int \ln{w_i(\cdot)} \, d\PP(\cdot)
\end{equation*}
for all $1 \le i \le N$ (from the assumptions alluded~to above it follows that $\ln{w_i(\cdot)}$ are in $L_1\bigl(\OFP, \RR\bigr)$ so, due to Birkhoff's ergodic theorem, the above equality holds for a.e.\ $\omega \in \Omega'$).  Put $w(n) :=
w(\theta^{n}\omega)$, $\mu(n) := \mu(\theta^{n}\omega)$, $b_i(n) := b_i(\theta^{n}\omega)$, $n = 0, 1, 2, \dots$ (recall that $w(n + 1) = \mu(n) w(n)$).

We have
\begin{equation*}
w_{i}(n + 1) = \frac{1}{\mu(n)} \sum\limits_{j=1}^{N} a_{ij} b_{j}(n) w_j(n).
\end{equation*}
Put
\begin{equation*}
\widetilde{w}_{i}(n) := \exp\Bigl( \frac{1}{n} \sum\limits_{k=0}^{n-1}
\ln{w_i(k)}\Bigr), \quad \widehat{w}_{i}(n) := \exp\Bigl( \frac{1}{n}
\sum\limits_{k=1}^{n} \ln{w_i(k)}\Bigr).
\end{equation*}
Further, let
\begin{equation*}
\widetilde{b}_{i}(n) := \exp\Bigl( \frac{1}{n} \sum\limits_{k=0}^{n-1}
\ln{b_i(k)}\Bigr), \quad \widetilde{\mu}(n) := \exp\Bigl( \frac{1}{n}
\sum\limits_{k=0}^{n-1} \ln{\mu(k)}\Bigr).
\end{equation*}
For any $i, j$ an application of the geometric\nobreakdash-\hspace{0pt}arithmetic mean
inequality gives that
\begin{equation*}
\exp\biggl( \frac{1}{n} \sum\limits_{k=0}^{n-1} \ln\Bigl(
\frac{1}{\mu(k)} b_j(k) \frac{w_{j}(k)}{w_{i}(k+1)} \Bigr) \biggl)
\le \frac{1}{n} \sum\limits_{k=0}^{n-1} \frac{1}{\mu(k)} b_j(k)
\frac{w_{j}(k)}{w_{i}(k+1)}.
\end{equation*}
For each $i$, by multiplying the above inequality by $a_{ij}$, $1 \le j \le N$, and adding the resulting inequalities one obtains, after some calculation, that
\begin{multline*}
\frac{1}{\exp\biggl( -\frac{1}{n} \sum\limits_{k=0}^{n-1} \ln{\mu(k)}
\biggr)} \sum\limits_{j=1}^{N} a_{ij} \exp\biggl( \frac{1}{n}
\sum\limits_{k=0}^{n-1} \ln{b_j(k)} \biggr) \frac{\widetilde{w}_j(n)}{\widehat{w}_i(n)}
\\
\le \frac{1}{n} \sum\limits_{k=0}^{n-1} \frac{1}{\mu(k)} b_j(k) \frac{\sum\limits_{j=1}^{N} a_{ij} w_j(k)}{w_i(k)} = 1,
\end{multline*}
that is,
\begin{equation*}
\sum\limits_{j=1}^{N} a_{ij} \widetilde{b}_j(n) \widetilde{w}_j(n) \le \widetilde{\mu}(n) \widehat{w}_i(n).
\end{equation*}
It is a consequence of standard estimates in ergodic theory that $\tilde{\mu}(n)$ are
bounded away from zero, uniformly in $n$.

Take a sequence $(n_m)_{m=1}^{\infty} \subset \NN$ such that $\lim\limits_{m \to \infty} n_m = \infty$ and $\lim\limits_{m\to\infty} \widetilde{w}_i(n_m) =: \overline{w}_i$ exist for all $1 \le i \le N$.  Since
\begin{equation*}
\hat{w}_i(n) = \widetilde{w}_j(n) \exp\biggl( \frac{1}{n}
\ln{\frac{w_i(n)}{w_i(0)}} \biggr),
\end{equation*}
it follows that $\lim\limits_{m\to\infty} \widehat{w}_i(n_m) =\overline{w}_i$ for all $1 \le i \le N$.

As $\lim\limits_{n \to \infty} \widetilde{\mu}(n) = e^{\Lambda}$, we have therefore found $\overline{w} \in \RR^{N}_{++}$ such that
\begin{equation*}
A \overline{B} \overline{w} \le e^{\Lambda} \overline{w}
\end{equation*}
holds.  By~\cite[Theorem (1.11) on p.~28]{BP}, the logarithm of the spectral radius of $A\overline{B}$ does not exceed the principal Lyapunov exponent $\Lambda$.

\part{Concluding Remarks}
\label{part:concluding}
\begin{remark}
{\em In the present survey a common feature is that it has been done from the point of view of \textbf{individual trajectories}: there is some distinguished solution (denoted by $v(t)$ or the like) the logarithmic growth rate of the norm of which serves as the principal Lyapunov exponent.  We can call this approach \textbf{dynamical}.

A variant of the dynamical approach rests on representing, via Birkhoff's ergodic theorem, the logarithmic growth rate of the distinguished solution with the integral over a measure space $\Omega$ of some computable quantity, as in formula~\ref{eq:ODE-Lambda-AP-2}.

\medskip
There can be a different approach, more \textbf{functional\nobreakdash-\hspace{0pt}analytic} in its spirit.  Namely, the original system gives rise to an {\em evolution semigroup\/} of bounded linear operators on the space $L_p(\RR, X)$, or $C_{0}(\RR, X)$, or others (see~\cite{ChicLat}).  Evolution semigroups generated by systems considered in the present survey possess positivity (order\nobreakdash-\hspace{0pt}preserving) properties, which can allow one to apply the spectral theory of such semigroups (see, e.g., \cite{BanArl}).  That approach is especially fruitful for time\nobreakdash-\hspace{0pt}periodic systems, see~\cite{Hess}, or~\cite{HShV2}.}
\end{remark}

\begin{remark}
{\em It is a recurring theme in the present survey that, for systems in which interactions between ''components'' (be it indices $i \in \{1, \dots, N\}$ or points $x \in D$) are independent of time, the principal Lyapunov exponent of the original system is not smaller than the principal eigenvalue of the time\nobreakdash-\hspace{0pt}averaged (in the periodic or almost periodic cases) or the state\nobreakdash-\hspace{0pt}averaged (in the random case) system (sometimes one can even prove that the inequality is sharp except some trivial cases).  As, in applications to biological populations, the positivity of the principal Lyapunov exponent ensures the permanence of the population, we have that \textbf{temporal variation in the growth rates enhances the capability of the population to persist}.  This is in marked contrast to what is known in population biology, see~\cite[Introduction]{Schreib}.}
\end{remark}
\begin{problem}
{\em \textbf{Develop a theory which would unify all the mentioned results, in the continuous\nobreakdash-\hspace{0pt}time as well as in the discrete\nobreakdash-\hspace{0pt}time cases, on estimating the principal Lyapunov exponent via averaging.}  (For a unified treatment of some PDEs, nonlocal dispersal equations and some systems of ODEs see~\cite{ShV}).}
\end{problem}

\section*{Acknowledgments}
This work was supported by the NCN grant Maestro 2013/08/A/ST1/00275.

The author is grateful to Rafael Obaya and Ana M. Sanz for sending him a preprint of~\cite{COS}, and to Sylvia Novo, Rafael Obaya and Ana M. Sanz for valuable discussions.

\end{document}